\documentclass[reqno]{amsart}
\usepackage[utf8x]{inputenc}  % this enables accented characters

\usepackage{setspace}
\usepackage[left=3.2cm, right=3.2cm, bottom=4cm]{geometry}                 

\usepackage{graphicx} % this enables figures
\usepackage{pgf,tikz}
\usetikzlibrary{arrows}
\usepackage{amssymb}
\usepackage{pdfsync}
\usepackage{mathrsfs}
\usepackage{hyperref} % this enables hyperlinks
\usepackage{verbatim} % this enables the {comment} environment
\usepackage{epstopdf}
\usepackage{bbm} 
\usepackage[colorinlistoftodos,prependcaption,textsize=tiny]{todonotes}
\usepackage{xargs}

\def\R{\mathbb{R}}

\def\N{\mathbb{N}}
\def\Z{\mathbb{Z}}
\def\Co{\mathbb{C}}
\def\H{\mathbb{H}}

\def\supp{{\rm supp}}

\newcommandx{\emanuel}[2][1=]{\todo[linecolor=green,backgroundcolor=green!25,bordercolor=black,#1]{#2}}

\newcommandx{\diogo}[2][1=]{\todo[linecolor=orange,backgroundcolor=orange!25,bordercolor=orange,#1]{#2}}

\newcommandx{\mateus}[2][1=]{\todo[linecolor=blue,backgroundcolor=blue!25,bordercolor=blue,#1]{#2}}

\newcommandx{\danger}[2][1=]{\todo[linecolor=red,backgroundcolor=red!25,bordercolor=blue,#1]{#2}}

\renewcommand{\d}{\text{\rm d}}
\newcommand{\one}{\mathbbm 1}

\newcommand{\pvector}[1]{% column vector
  \begin{pmatrix}
    #1
  \end{pmatrix}} %
\newcommand{\ddirac}[1]{% Dirac's delta
  \,\boldsymbol{\delta}\!\pvector{#1}\!} %
\newcommand{\jp}[1]{\langle{#1}\rangle}

\newtheorem{theorem}{Theorem}

\newtheorem{proposition}[theorem]{Proposition}
\newtheorem{lemma}[theorem]{Lemma}

\numberwithin{equation}{section}

\title[Extremizers for adjoint Fourier restriction on hyperboloids]{Extremizers for adjoint Fourier restriction on hyperboloids: the higher dimensional case}

\author[Carneiro]{Emanuel Carneiro}
\author[Oliveira e Silva]{Diogo Oliveira e Silva}
\author[Sousa]{Mateus Sousa}
\author[Stovall]{Betsy Stovall}

\address{
IMPA - Instituto de Matem\'{a}tica Pura e Aplicada\\
Rio de Janeiro - RJ, Brazil, 22460-320.}
\email{carneiro@impa.br}
\address{
        University of Birmingham\\
        Edgbaston, Birmingham, B15 2TT, England.}
\email{d.oliveiraesilva@bham.ac.uk}

\address{Ludwig-Maximilans Universität München \\
Theresienstr. 39, 80333 München, Germany.}
\email{sousa@math.lmu.de}

\address{University of Wisconsin--Madison, 480 Lincoln Drive, Madison, Wisconsin, USA, 53706.}
\email{stovall@math.wisc.edu}

\begin{document}

\begin{abstract} 
We prove that in dimensions $d \geq 3$, the non-endpoint, Lorentz-invariant $L^2 \to L^p$ adjoint Fourier restriction inequality on the $d$-dimensional hyperboloid $\H^d \subseteq \R^{d+1}$ possesses maximizers.  The analogous result had been previously established in dimensions $d=1,2$ using the convolution structure of the inequality at the lower endpoint (an even integer); we obtain the generalization by using tools from bilinear restriction theory.  
\end{abstract}

\maketitle
%\tableofcontents

\section{Introduction} 

\subsection{Setup} In this note we continue the study initiated in \cite{COSS17, Qu15} on sharp Fourier restriction theory on hyperboloids. 
Let us start by recalling the basic terminology and the main definitions.

\smallskip

Throughout this work we adopt the following normalization for the Fourier transform in $\R^{d+1}$:
\begin{equation}\label{NormalizeFT}
\widehat{g}(\zeta)=\int_{\R^{d+1}} e^{-iz\cdot\zeta}\, g(z)\,\d z.
\end{equation}
If $\xi \in \R^d$, we define $\langle \xi \rangle := (1 + |\xi|^2)^{\frac12}$. The hyperboloid $\H^d \subset \R^{d+1}$ is the surface defined by\footnote{A simple rescaling argument transfers all the results of this paper to the hyperboloids $\H^d_s=\big\{(\xi,\tau)\in\R^d\times\R: \tau=(s^2+|\xi|^2)^{\frac12}\big\}$, $s > 0$.}
\begin{equation*}
\H^d=\big\{(\xi,\tau)\in\R^d\times\R: \tau=\langle \xi \rangle\big\},
\end{equation*}
and comes equipped with the Lorentz-invariant measure
\begin{equation}\label{defsigma}
\d\sigma(\xi,\tau)=\ddirac{\tau-\jp{\xi}}\frac{\d \xi\,\d \tau}{\langle \xi \rangle},
\end{equation}
which is defined by duality on an appropriate dense class of functions via the identity
$$\int_{\H^d}\varphi(\xi,\tau)\,\d\sigma(\xi,\tau)=\int_{\R^d} \varphi(\xi,\langle \xi \rangle)\frac{\d \xi}{\langle \xi \rangle}.$$
The {\it Fourier extension operator} on $\H^d$ (or adjoint Fourier restriction operator) is given by
\begin{equation}\label{eq:DefExtensionOp}
T(f)(x,t):=\int_{\R^d} e^{ix\cdot \xi} e^{it\langle \xi \rangle} f(\xi)\frac{\d \xi}{\langle \xi \rangle},
\end{equation}
where $(x,t)\in\R^d\times\R$ and $f$ belongs to the Schwartz class in $\R^d$.
Throughout this note we identify a function $f:\H^d\to\Co$ with a complex-valued function defined on $\R^d$. The norm in $L^p(\H^d) = L^p(\H^d, \sigma)$ is then given by
$$\|f\|_{L^p(\H^d)}=\left(\int_{\R^d} |f(\xi)|^p \frac{\d \xi}{\langle \xi \rangle}\right)^{\frac1p}.$$
With the Fourier transform normalized as in \eqref{NormalizeFT}, note that
\begin{equation}\label{TintermsofHat}
T(f)(x,t)=\widehat{f\sigma}(-x,-t).
\end{equation}
The seminal work of Strichartz \cite[Theorem 1, Cases III (b)(c)]{St77} establishes the estimate
\begin{equation}\label{ExtensionInequality}
\|T(f)\|_{L^p(\R^{d+1})}\leq {\bf H}_{d,p} \, \|f\|_{L^2(\H^d)}\,,
\end{equation}
with a finite constant ${\bf H}_{d,p}$ (independent of $f$), provided that
  \begin{equation}\label{AdmissibleRange}
   \begin{cases}
    6 \leq p< \infty, \text{ if } d=1;\\
  \frac{2(d+2)}{d} \leq p\leq  \frac{2(d+1)}{d-1}, \text{ if } d\geq 2.
      \end{cases} 
\end{equation}
We reserve the symbol ${\bf H}_{d,p}$ for the optimal constant
\begin{equation}\label{Def_Sharp_Const}
{\bf H}_{d,p}:=\sup_{0\neq f\in L^2(\H^d)} \frac{\|T(f)\|_{L^p(\R^{d+1})}}{\|f\|_{L^2(\H^d)}},
\end{equation}
and say that a nonzero function $f\in L^2(\H^d)$ is an {\it extremizer} of \eqref{ExtensionInequality} if it realizes the supremum in \eqref{Def_Sharp_Const}, and we call a nonzero sequence  $\{f_n\} \subset L^2(\H^d)$ an {\it extremizing sequence} of \eqref{ExtensionInequality} if the ratio ${\|T(f_n)\|_{L^p(\R^{d+1})}}/{\|f_n\|_{L^2(\H^d)}}$ converges to ${\bf H}_{d,p}$ {as $n\to\infty$}.

\subsection{Main theorem} The first result to address the sharp form of \eqref{ExtensionInequality} is due to Quilodr\'{a}n \cite{Qu15}, in which he computes the exact values of ${\bf H}_{d,p}$ in the endpoint cases $(d,p) = (2,4),(2,6)$ and $(3,4)$, and establishes the non-existence of extremizers in these cases.\footnote{By contrast, the recent work \cite{Qu18} establishes existence of extremizers for the endpoint $L^2$ to $L^4$ adjoint Fourier restriction inequality on the one-sheeted hyperboloid in dimension 4.} A crucial element of his proof is the fact that the Lebesgue exponents $p$ under consideration are even integers, a fact that allows one to use the convolution structure of the problem via an application of Plancherel's theorem. In \cite{Qu15}, Quilodr\'{a}n also raises two interesting questions: What is the value of the sharp constant at the endpoint $(d,p) = (1,6)$ (the remaining case with $p$ even); and do extremizers exist in the non-endpoint cases.

\smallskip

The precursor \cite{COSS17} of the present work contains two main results. The first result \cite[Theorem 1]{COSS17} is the explicit computation of the optimal constant ${\bf H}_{d,p}$ in the case $(d,p) = (1,6)$ and the proof that extremizers do not exist in this case. The second result \cite[Theorem 2]{COSS17} establishes the existence of extremizers in all non-endpoint cases of \eqref{ExtensionInequality} in dimensions $d \in \{1,2\}$. The proof of the latter result is obtained  by establishing {that} extremizing sequences converge modulo certain symmetries of the problem. {In the present case}, by a {\it symmetry} we mean an operator $S:L^2(\H^d)\rightarrow L^2(\H^d)$ such that 
$$\|Sf\|_{L^2(\H^{d})}=\|f\|_{L^2(\H^{d})}
\text{ and }
\|T(Sf)\|_{L^p(\R^{d+1})}=\|T(f)\|_{L^p(\R^{d+1})}.$$\noindent {Such an operator} can shift the mass of sequences and destroy strong convergence while still mantaining its extremizing properties, hence the study of these symmetries is fundamental. In the {case of the hyperboloid}, one has to account for the action of the Lorentz group and space-time modulations (and their compositions), which we introduce in more detail in the next section. In \cite{COSS17}, the convergence is obtained via a direct and self-contained approach that explores the convolution structure of the problem at the lower endpoint (which is an even integer in these low dimensions). The drawback of this particularly simple proof is that it does not work in the higher dimensional cases $d \geq 3$.

\smallskip

In this note we return to this problem and extend the result of \cite[Theorem 2]{COSS17} to dimensions $d\geq 3$. Our main result is the following.

\begin{theorem}\label{Thm1}
Let $d\geq 3$.
Extremizers for inequality \eqref{ExtensionInequality} exist if $\frac{2(d+2)}{d} < p < \frac{2(d+1)}{d-1}$. In fact, given any extremizing sequence $\{f_n\}$, there {exist} symmetries $S_n$ such that $\{S_nf_n\}$ converges in $L^2(\H^d)$ to an extremizer $f$, {after passing to a subsequence}.
\end{theorem}

The main new ingredient of the proof, when compared to that of \cite[Theorem 2]{COSS17}, is the use of machinery from bilinear restriction theory to obtain a refined version of inequality \eqref{ExtensionInequality}. 
As in \cite{Ca17, COSS17}, we exploit the fact that the hyperboloid is well approximated by the paraboloid and the cone. 
The geometric construction underlying the bilinear restriction machinery accounts for this fact:  in some sense, it interpolates between the two endpoint cases, which we will refer to as the {\it elliptic} and the {\it conic} regimes, respectively.

\smallskip

Estimates for Fourier extension operators are related to estimates for dispersive partial differential equations. In our case, the extension operator $T$ defined in \eqref{eq:DefExtensionOp} is related to the Klein--Gordon equation $\partial_t^2u = \Delta_x u -u$ for $(x,t)\in\R^d\times\R$. Defining the (half) {\it Klein--Gordon propagator} as
\begin{equation*}
    e^{it\sqrt{1-\Delta}}g(x):= \frac{1}{(2\pi)^d}\int_{\R^d} e^{i x\cdot\xi}\,e^{it\langle\xi\rangle}\,\widehat{g}(\xi)\,\d\xi,
\end{equation*}
one readily sees that 
\begin{equation}\label{July19_12:43am}
T(f)(x,t)=(2\pi)^{d} \,e^{it\sqrt{1-\Delta}}\, g(x),
\end{equation} 
with $\widehat{g}(\xi)=\langle \xi \rangle^{-1}f(\xi)$.
Therefore, inequality \eqref{ExtensionInequality} can be restated as 
\begin{equation*}
   \|e^{it\sqrt{1-\Delta}}g\|_{ L_{x,t}^p(\R^d\times\R)}\leq (2\pi)^{-d} \,{\bf H}_{d,p}\,\|g\|_{H^{\frac{1}{2}}(\R^d)},
\end{equation*}
where for $s \geq 0$ we denote by $H^s(\R^d)$ the nonhomogeneous Sobolev space, defined as
$$H^{s}(\R^d)=\Big\{g\in L^2(\R^d):~\|g\|_{H^s(\R^d)}^2:=\int_{\R^d}|\widehat{g}(\xi)|^2\jp{\xi}^{2s}\d\xi<\infty\Big\}.$$
The reader should keep in mind this equivalent formulation, since some of the results we quote from \cite[Section 6]{COSS17} are stated in terms of the Klein--Gordon propagator. 

\smallskip

Extremal problems related to Fourier restriction theory have garnished a lot of attention in recent years, and a large body of {work} has emerged. Several authors have investigated the interface between bilinear restriction theory and these extremal questions, {both from the restriction side and the partial differential equations point of view}. Here we mention the works \cite{BV07,Ca17,FLS16,FS17,JSS,KSV12,Ra12}, all of which deal with these connections. Many other authors have contributed to the development of the area, and we refer the reader to \cite{COSS17} for an exposition of related literature on sharp Fourier restriction theory.

\subsection{Outline} We discuss the Lorentz symmetry of the problem in Section \ref{sec:Prelims}, where we also establish an annular decoupling inequality which implies a modest  gain of control over extremizing sequences.
The actual proof starts in Section \ref{sec:ang_rest}, with a simple but useful argument that allows us to restrict the angular support of the functions under consideration.
In Section \ref{sec:Caps}, we describe a geometric decomposition of space into {\it caps} and {\it sectors}, and the corresponding bilinear restriction estimates that will play a key role in the analysis. 
As in \cite{COSS17}, the crux of the matter is the construction of a {\it distinguished region}, i.e. the lift of a cap or a sector to the hyperboloid that contains a positive universal proportion of the total mass in an extremizing sequence. 
We establish this fact via a refined Strichartz inequality, formulated as Theorem \ref{Sharpened} and proved in Section \ref{sec:RefinedStrichartz}.
Once the existence of a special region has been established in dimensions $d \geq3$, the proof of Theorem \ref{Thm1} is finished by invoking the concentration-compactness material of \cite[Section 6]{COSS17}, which was already tailor-made to receive the input in any dimension. The details are outlined in Section \ref{sec:CC}.

\subsection{Notation} {\it Universal} quantities will be allowed to depend only on the dimension $d$ and the Lebesgue exponent $p$.
In a similar spirit, given $A,B\geq 0$, we write $A\simeq B$ (resp. $A \lesssim B$) and say that $A,B$ are {\it comparable} if there exists a finite constant $C=C(d,p)>0$, such that $\frac1C B\leq A\leq CB$ (resp. $A \leq CB$).
A number $N$ is said to be {\it dyadic} if it is an integral power of 2, i.e. $N\in2^{\Z}$.

%%%%%%%%%%%%%%%%%%%%%%%%%%%%%%%%%%%%%%%%%%%%%%%%%%%%%%%%%%%%%%%%%%%%%%%%%%%%%%%%%%%%%%%%%%%%%%%%%%%%%%%%%%%%%%%%%%%%%%%%%%%%%%%%%%%%%%%%%%%%%%%%

\section{Preliminaries}\label{sec:Prelims}
\subsection{Lorentz boosts}
The Lorentz group, denoted $\mathcal{L}$, is defined as the group of invertible linear transformations in $\R^{d+1}$ that preserve the bilinear form 
$(x,y)\in\R^{d+1}\times\R^{d+1}\mapsto x\cdot Jy,$ where $J=\text{diag}(-1,\ldots,-1,1)$.
In particular, if $L\in\mathcal{L}$, then $|\det L|=1$.
Denote the subgroup of $\mathcal{L}$ that preserves $\mathbb{H}^d$ by $\mathcal{L}^+$.
A one-parameter subgroup of $\mathcal{L}^+$ is $\{L^t\}_{t\in(-1,1)}$, where the linear map $L^t:\R^{d+1}\to\R^{d+1}$ is defined via
$$L^t(\xi_1,\ldots,\xi_d,\tau)=\left(\frac{\xi_1+t\tau}{\sqrt{1-t^2}},\xi_2,\ldots,\xi_d,\frac{\tau+t\xi_1}{\sqrt{1-t^2}}\right).$$ 
Given an orthogonal matrix $A\in\text{O}(d)$, the map $(\xi,\tau)\mapsto(A\xi,\tau)$ belongs to $\mathcal{L}^+$.
A way to parametrize more general
Lorentz boosts is as follows.
Given a frequency parameter $\nu\in\R^d$, we define the Lorentz boost in the direction $\nu$ as
\begin{equation}\label{eq:Lorentz}
L_\nu(\xi,\tau):=
(\xi^\perp+\jp{\nu}\xi^\parallel-\nu \tau,\jp{\nu}\tau-\nu\cdot \xi).
\end{equation}
Here $\xi^\perp$ and $\xi^\parallel$  denote the components of $\xi$ which are orthogonal and parallel to $\nu$, respectively.
The boost $L_\nu$ preserves space-time volume since its determinant is one,
and acts on $\R^d$ via
\begin{equation}\label{eq:RdLorentz}
L_\nu^\flat(\xi):=\xi^\perp+\jp{\nu}\xi^\parallel-\nu \jp{\xi}.
\end{equation}
Note that $L_{\nu}^{-1}=L_{-\nu}$, and likewise $(L_{\nu}^\flat)^{-1}=L_{-\nu}^\flat$.
We also have that $L_\nu(\nu,\jp{\nu})=(0,1)$, and correspondingly $L_\nu^\flat(\nu)=0$.
For $p\in[1,\infty]$, $L\in\mathcal{L}^+$, and $f\in L^p(\H^d)$, define the composition $L^*f=f\circ L$.
Then one easily checks that
$$\|L^*f\|_{L^p(\H^{d})}=\|f\|_{L^p(\H^{d})}
\text{ and }
\|T(L^*f)\|_{L^p(\R^{d+1})}=\|T(f)\|_{L^p(\R^{d+1})}.$$

\subsection{Annular decoupling}\label{sec:AnnularDec} 
The extension operator $T$ defined in \eqref{eq:DefExtensionOp} satisfies more general mixed-norm estimates of which \eqref{ExtensionInequality} is a particular case.
As pointed out in \cite{KO11} and the references therein, the inequality
\begin{equation}\label{MixedNormKGStrichartz}
\|T(f)\|_{L^q_t L^r_x(\R^{d+1})}\lesssim \|\jp{\xi}^{\frac1q-\frac1r} f\|_{L^2_\xi(\H^d)}
\end{equation}
holds, provided 
$q\in[2,\infty]$,
$r\in[2,2d/(d-2)]$
($r\in[2,\infty]$ if $d\in\{1,2\}$), and
$$\frac2q+\frac{d-1+\theta}r=\frac{d-1+\theta}2,\;(q,r)\neq(2,\infty),$$
for some $\theta\in[0,1]$.
A pair $(q,r)$ of Lebesgue exponents satisfying these conditions will be referred to as an {\it admissible pair}.
Certain instances of inequality \eqref{MixedNormKGStrichartz} together with a variant of the Littlewood--Paley decomposition yield an annular decoupling inequality which we now prove.

\smallskip

We will use a dyadic frequency decomposition.
To implement it,
let $N\geq 1$ be a dyadic number. 
Given $f\in L^2(\H^d)$, we denote by $f_N$ the smoothed out restriction of $f$ to frequencies $|\xi|\simeq N$. More precisely, fix a smooth radial bump function $\psi:\R^d\to[0,1]$ supported in the ball $\{\xi\in\R^d:|\xi|\leq\frac{11}{10}\}$ and equal to 1 on the unit ball $\{\xi\in\R^d:|\xi|\leq1\}$, and define
\begin{displaymath}
f_N(\xi) := \left\{ \begin{array}{ll}
\psi(\xi){f}(\xi), & \textrm{if }N=1,\\
\big(\psi(\frac{\xi}N)-\psi(\frac{2\xi}N)\big){f}(\xi), & \textrm{if $N>1$.}
\end{array} \right.
\end{displaymath}
Note that $\supp(f_1) \subseteq \{\xi\in\R^d:|\xi| \leq 2\}$ and $\supp(f_N) \subseteq \{\xi\in\R^d:\tfrac N2 \leq |\xi|\leq 2N\}$, for $N >1$. 
The following annular decoupling is in the spirit of \cite{JSS, KSV12}.
\begin{proposition}\label{annulardec}
Let $d\geq 3$ and $\frac{2(d+2)}{d} \leq p \leq  \frac{2(d+1)}{d-1}$.
 Then
\begin{equation}\label{Annular_decoupling}
\|T(f)\|^p_{L^p(\R^{d+1})}\lesssim 
\sup_{N\in 2^{\Z_{\geq 0}}} \|T(f_N)\|_{L^p(\R^{d+1})}^{p-2} 
\|f\|_{L^{2}(\H^d)}^2,
\end{equation}
for every $f\in L^2(\H^d)$.
\end{proposition}

\begin{proof}
By the Littlewood--Paley square function estimate, we have that
\begin{equation}\label{eq:LPConsequence}
\|T(f)\|_{L_{x,t}^p}^p
\simeq\Big\|\Big(\sum_N |T (f_N)|^2\Big)^{\frac12}\Big\|_{L^p_{x,t}}^p.
\end{equation}
Indeed,  $\mathcal{F}_x[T(f)](\xi,t)=e^{it\langle\xi\rangle}\langle\xi\rangle^{-1}f(\xi)$, where $\mathcal{F}_x$ denotes the Fourier transform  in the variable $x\in\R^d$.
Standard Littlewood--Paley theory yields
$$\|T(f)(\cdot, t)\|_{L^p_x}^p\simeq\Big\|\Big(\sum_N |T(f_N)(\cdot,t)|^2\Big)^{\frac12}\Big\|_{L^p_x}^p,$$
for each fixed $t\in\R$.
Estimate \eqref{eq:LPConsequence} then follows from integration in the time variable $t$.

Since $d\geq 3$, we have that $\frac p2\leq 2$, and thus the sequence space embedding $\ell^{\frac p2}\hookrightarrow\ell^2$ implies
$$\Big(\sum_N |T (f_N)|^2\Big)^{\frac 12}\leq \Big(\sum_N |T (f_N)|^{\frac p2}\Big)^{\frac 2p}.$$
We can estimate
\begin{equation}\label{triangleineq}
\|T(f)\|_{L^p_{x,t}}^p
\lesssim\int_{\R^{d+1}} \Big(\sum_N |T(f_N)|^{\frac p2}\Big)^2
=\int_{\R^{d+1}} \sum_{N,M} |T (f_M) T (f_N)|^{\frac p2}
\lesssim \sum_{M} \sum_{N\leq M} \|T (f_M)T (f_N)\|_{L^{\frac p2}_{x,t}}^{\frac p2},
\end{equation}
where the last inequality follows from Fubini's theorem and symmetry.
We control each of the summands of the right-hand side of \eqref{triangleineq} using the mixed-norm estimates \eqref{MixedNormKGStrichartz}.
With this purpose in mind, fix admissible pairs $(q_0,r_0)$ and $(q_1,r_1)$ with $q_1<p<q_0$ and $r_0<p<r_1$, which additionally satisfy 
\begin{equation}\label{eq:MixedNormExponents}
\frac2p=\frac1{q_0}+\frac1{q_1}=\frac1{r_0}+\frac1{r_1}.
\end{equation}
Then, invoking H\"older's inequality twice, we have that
\begin{align*}
\|T (f_M)T (f_N)\|_{L^{\frac p2}}^{\frac p2}
&\leq\|T (f_M)\|_{L^p}^{\frac p2-1}\|T (f_N)\|_{L^p}^{\frac p2-1}\|T (f_M)T (f_N)\|_{L^{\frac p2}}\\
&\leq\|T (f_M)\|_{L^p}^{\frac p2-1}\|T (f_N)\|_{L^p}^{\frac p2-1}\|T (f_M)\|_{L_t^{q_0}L_x^{r_0}}\|T (f_N)\|_{L_t^{q_1}L_x^{r_1}}\\
&\lesssim
\|T (f_M)\|_{L^p}^{\frac p2-1}\|T (f_N)\|_{L^p}^{\frac p2-1}
\|\langle\xi\rangle^{\frac1{q_0}-\frac1{r_0}} f_M\|_{L^2(\H^d)}
\|\langle\xi\rangle^{\frac1{q_1}-\frac1{r_1}} f_N\|_{L^2(\H^d)}.
\end{align*}
where the last line is a consequence of \eqref{MixedNormKGStrichartz}.
Since $\langle\xi\rangle\simeq M$ inside the support of $f_M$, and similarly for $f_N$, from this and  \eqref{eq:MixedNormExponents} it follows that
$$\|T (f_M)T (f_N)\|_{L^{\frac p2}}^{\frac p2}
\lesssim
\Big(\frac NM\Big)^{\frac1{q_1}-\frac1{r_1}}
\|T (f_M)\|_{L^p}^{\frac p2-1}\|T (f_N)\|_{L^p}^{\frac p2-1}
\|f_M\|_{L^2(\H^d)}
\|f_N\|_{L^2(\H^d)}.
$$
Going back to \eqref{triangleineq} and noting that $\frac1{q_1}-\frac1{r_1}>0$, 
we use H\"older's inequality and the elementary estimate $2ab\leq a^2+b^2$ with $a=\|f_N\|_{L^2(\H^d)}$ and $b=\|f_M\|_{L^2(\H^d)}$,
and sum a geometric series to  finally conclude that
\begin{align*}
\sum_{M} \sum_{N\leq M} \|T (f_M)T (f_N)\|_{L^{\frac p2}}^{\frac p2}
&\leq\sup_N \|T (f_N)\|_{L^p}^{p-2}
\sum_{M} \sum_{N\leq M} 
\Big(\frac NM\Big)^{\frac1{q_1}-\frac1{r_1}}
\|f_M\|_{L^2(\H^d)}
\|f_N\|_{L^2(\H^d)}\\
&\lesssim\sup_N \|T (f_N)\|_{L^p}^{p-2}
\|f\|^2_{L^2(\H^d)}.
\end{align*}
This finishes the proof of the proposition.
\end{proof}

%%%%%%%%%%%%%%%%%%%%%%%%%%%%%%%%%%%%%%%%%%%%%%%%%%%%%%%%%%%%%%%%%%%%%%%%%%%%%%%%%%%%%%%%%%%%%%%%%%%%%%%%%%%%%%%%%%%%%%%%%%%%%%%%%%%%%%%%%%%%%%%%
  
\section{Beginning of the proof: angular restriction} \label{sec:ang_rest}

Let $\{f_n\}_{n\in\N} \subset L^2(\mathbb{H}^d)$ be an extremizing sequence for \eqref{ExtensionInequality}. We may assume that $\|f_n\|_{L^2(\mathbb{H}^d)} = 1$ and that $\|T(f_n)\|_{L^p(\R^{d+1})} \to {\bf H}_{d,p}$ as $n \to \infty$. Recall from the Introduction that each $f_n$ is regarded as a function on $\R^d$. 
Given $K\in\N$,  consider a finite partition of the unit sphere $\mathbb{S}^{d-1} =\{\xi\in \R^d: |\xi|=1\}$ into $K$ disjoint regions,
$$\mathbb{S}^{d-1} = \bigcup_{k=1}^{K} C_k^*.$$
Given a function $f:\R^d \to \mathbb{C}$, let $f^{(k)} := f \ \one_{R_k}$, where $R_k = \{\xi \in \R^d : \xi/|\xi| \in C_k^*\}$. In this way we split $\R^d$ into $K$ angular sectors. The triangle inequality implies  
$$\|T(f_n)\|_{L^p(\R^{d+1})} \leq \sum_{k=1}^{K} \|T(f_n^{(k)})\|_{L^p(\R^{d+1})}.$$
Observe that, possibly after extraction of a subsequence, there exists $k_0 \in \{1,2,\ldots, K\}$ such that $\{f_n^{(k_0)}\}_{n\in\N}$ is a {\it quasi-extremizing sequence} for \eqref{ExtensionInequality}. 
By this we mean that $\|f_n^{(k_0)}\|_{L^2(\mathbb{H}^d)} \leq 1$, and 
\begin{equation}\label{quasi-ext}
\|T(f_n^{(k_0)})\|_{L^p(\R^{d+1})}\geq \delta_1,
\end{equation}
for every $n\in\N$ and some universal $\delta_1 >0$ (we may take for instance $\delta_1 = \frac{{\bf H}_{d,p}}{2K}$).

\smallskip

Under these circumstances, we will establish the existence of 
a universal ball $B \subset \R^d$ centered at the origin, 
a universal $\delta_2 >0$, 
and a sequence of Lorentz transformations $\{L_n\}_{n\in\N}$ such that 
$$\|L_n ^* f_n^{(k_0)}\|_{L^2(B)} \geq \delta_2,$$
for every $n\in\N$.
This naturally implies 
$$\|L_n^* f_n\|_{L^2(B)} \geq \delta_2,$$
for every $n\in\N$.
The latter inequality is of the sort which is required in order to invoke the machinery from \cite[Section 6]{COSS17} and conclude the proof of Theorem \ref{Thm1}.

\smallskip

Throughout the upcoming Sections \ref{sec:Caps} and \ref{sec:RefinedStrichartz} we will thus assume that our functions are supported in a small angular region $R_1$
(the corresponding $C_1^* \subset \mathbb{S}^{d-1}$ is described at the beginning of Section \ref{sec:Caps}).
Henceforth, such functions will be referred to as {\it admissible}.

%%%%%%%%%%%%%%%%%%%%%%%%%%%%%%%%%%%%%%%%%%%%%%%%%%%%%%%%%%%%%%%%%%%%%%%%%%%%%%%%%%%%%%%%%%%%%%%%%%%%%%%%%%%%%%%%%%%%%%%%%%%%%%%%%%%%%%%%%%%%%%%%

\section{Caps, sectors and bilinear estimates}\label{sec:Caps}
As mentioned in the Introduction, one of the key ingredients in the proof of Theorem \ref{Thm1} is the use of tools from bilinear restriction theory. Classical works on the topic include \cite{Ta03, TVV98, Wo01}.
\smallskip

In this section, we define the appropriate geometric regions and the notion of separation between them, and establish the bilinear restriction estimates that will be of relevance in the sequel.

\subsection{Definition of dyadic regions}  
Let $d \geq 3$ be a fixed dimension.
Consider the $(d-1)$-dimensional cube
$$C_1 = \{\eta = (\eta_1, \eta_2, \ldots, \eta_{d-1}) \in \R^{d-1} : |\eta_i| \leq \ell, \; i=1,2,\ldots,d-1\}$$
of sidelength $2\ell$ centered at the origin. The quantity $\ell < \tfrac14$ is a small fixed number which depends only on the dimension $d$, and shall be appropriately chosen in due course. Given a dyadic number $M \in 2^{\Z_{\leq0}}$, let $\Gamma_M$ denote the usual dyadic decomposition of the cube $C_1$ into cubes of sidelength $2\ell M$ on $\R^{d-1}$. In particular, $\Gamma_1 = \{C_1\}$, and $\Gamma_M$ consists of $M^{-(d-1)}$ {\it essentially disjoint} cubes
(i.e. the intersection of any two distinct cubes is a Lebesgue null-set).
 Let $* \ : C_1 \to \mathbb{S}^{d-1}$ be the lift of a point in $C_1$ to a point in the unit sphere $\mathbb{S}^{d-1} \subset \R^d$, defined via
$$\eta^* = (\eta, (1 - |\eta|^2)^{\frac12}).$$
For each cube $Q \in \Gamma_M$, let 
$$Q^*  = \{\eta^* : \eta \in Q\}$$
denote the lift of the cube $Q$, and let $\Gamma^*_M$ denote the collection of the lifted cubes of $\Gamma_M$.

\smallskip

For the purposes of the present construction, we may think of distances in $C^*_1 \subset \mathbb{S}^{d-1} \subset \R^d$ as being almost the same as Euclidean distances in $C_1 \subset \R^{d-1}$. More precisely, given any constant $\varepsilon_1 >0$, we may choose $\ell = \ell(d,\varepsilon_1)>0$ sufficiently small, such that 
\begin{equation}\label{Sphere correction}
|\eta - \zeta|\leq\text{dist}(\eta^*, \zeta^*) \leq (1 + \varepsilon_1) \,|\eta - \zeta|
\end{equation}
for all $\eta, \zeta \in C_1 \subset \R^{d-1}$. Here $| \cdot | $ denotes Euclidean distance in $\R^{d-1}$, and dist$( \cdot, \cdot)$ denotes the geodesic distance on $\mathbb{S}^{d-1} \subset \R^d$. We may take for instance $\varepsilon_1 = \frac1{100}$.

\smallskip

Given $N\in2^{\Z_{>0}}$, define the restricted dyadic annulus
\begin{gather}\label{eq:DefAnnulus}
\mathcal{A}_N:=\big\{\xi\in\R^d: \tfrac12 N\leq|\xi|\leq 2N \ {\rm and} \ \xi/|\xi| \in C_1^*\big\},
\end{gather}
and set $\mathcal{A}_1:=\{\xi\in\R^d: |\xi|\leq 2  \ {\rm and} \ \xi/|\xi| \in C_1^*\}$. 

\smallskip

Given $N\in2^{\Z_{\geq0}}$, let $r\in 2^\Z$ be such that $0<r\leq N$.
If $0<r\leq1$, then we further decompose the restricted annulus $\mathcal{A}_N$ into an essentially disjoint union of regions 
\begin{equation}\label{height1}
\mathcal{A}_N^{(j)}:=\big\{\xi\in\mathcal{A}_N: \tfrac12 N(1+3j r)\leq|\xi|\leq \tfrac12 N(1+3(j+1)r)\big\},
\end{equation}
for $j\in J:=\{0,1,\ldots,r^{-1}-1\}$.
If $1< r\leq N$, then we unify the notation below by letting $J = \{0\}$ and $\mathcal{A}_N^{(0)}:=\mathcal{A}_N$.
In both cases we then have that $\# J=\max\{1,r^{-1}\}$.

\smallskip

Given $N\in2^{\Z_{\geq0}}$, and $r\in 2^\Z$ such that $0<r\leq N$, let $M = r/N$ and consider
$$\mathcal{D}_{N,r}
:=\big\{\kappa_{N,r}^{j,k}: (j,k)\in J\times \{1,2,\ldots, M^{-(d-1)}\}\big\},$$
where the regions $\kappa=\kappa_{N,r}^{j,k}$ are defined as
\begin{equation}\label{eq:DefTau}
\kappa_{N,r}^{j,k}:=\big\{\xi\in \mathcal{A}_N^{(j)}: \xi/|\xi| \in Q_k^*\big\},
\end{equation}
and $Q_k^*$ is a cube in the collection $\Gamma_M^*$.
The center of a region $\kappa=\kappa_{N,r}^{j,k}$ as in \eqref{eq:DefTau} is \mbox{defined to be}
\begin{equation}\label{eq:DefCenter}
c(\kappa):=\tfrac12 N\big(1+3\min\{1,r\} (j+\tfrac12)\big)\,\omega_k^*,
\end{equation}
where $\omega_k^*\in\mathbb{S}^{d-1}$ is the lift of the center $ \omega_k$ of the cube $Q_k \in \Gamma_M$.

\smallskip

If $0<r\leq 1$, then an element of $\mathcal{D}_{N,r}$ is called an {\it $r$-cap at scale $N$}.
If $1<r\leq N$, then an element of $\mathcal{D}_{N,r}$ is called an {\it $r$-sector at scale $N$}.
The Lebesgue measure of an $r$-cap  at scale $N$ is comparable to $Nr^d$, and 
the Lebesgue measure of an $r$-sector at scale $N$ is comparable to $Nr^{d-1}$.

\smallskip

By a {\it region} we will continue to mean a set which is either a cap or a sector. 
For fixed $N,r$, the regions in $\mathcal{D}_{N,r}$ are essentially disjoint. 
If $r < N$, then each $\kappa \in \mathcal{D}_{N,r}$ is contained in a unique $\kappa^\circ \in \mathcal{D}_{N,2r}$, and we refer to $\kappa^\circ$ as the {\it parent} of $\kappa$.
In a similar spirit, each $\kappa \in \mathcal{D}_{N,r}$ has either $2^{d-1}$ or $2^{d}$ {\it children}, according to the change of regime when $r=1$.

\smallskip

The construction outlined above can be regarded as a hybrid between a dyadic decomposition on $\R^{d}$ (caps) and on $\R^{d-1}$ (sectors), and is convenient to treat the elliptic and conic regimes in a unified way.

\subsection{Separated regions}\label{sec:separated} 

We call two regions {\it adjacent} if their closures intersect, possibly at boundary points.
We say that two regions $\kappa,\kappa'\in \mathcal{D}_{N,r}$ are {\it separated}, and write $\kappa\sim\kappa'\in \mathcal{D}_{N,r}$, if $\kappa,\kappa'$ are not adjacent, their parents are not adjacent, their $2$-parents (i.e. grandparents) are not adjacent, \ldots, their $(d-1)$-parents are not adjacent, and their $d$-parents are adjacent. Naturally, this assumes that $r \leq N/2^d$, so that $\kappa,\kappa'$ indeed have ancestors up to the $d$-th generation. 
The main reason why we climb up $d$ degrees in the genealogical tree when defining separation is to ensure that certain  naturally arising geometric regions which contain $\kappa, \kappa'$  are also ``separated''. In fact, as will become clear from the proof below, around $k$ generations up in the tree with $k \simeq \log_2 d$ would morally suffice.

If $\kappa,\kappa'\in \mathcal{D}_{N,r}$ are separated regions, then either: (i) the angular distance between $c(\kappa)$ and $c(\kappa')$ (which is $\simeq N |\omega_k^* - \omega_{k'}^*|$) is comparable to $r$; or (ii) the radial distance between $c(\kappa)$ and $c(\kappa')$ is comparable to $N r$. Note that option (ii) is only available if $0 < r < 2^{-d}$. 

\smallskip

Defining the regions and the separation between them in this way, we ensure that the union in the forthcoming expression \eqref{eq:Whitney} is essentially disjoint, an important step in the proof of the refined Strichartz estimate.

\subsection{Bilinear estimates} If $\kappa\in\mathcal{D}_{N,r}$ is a dyadic region as defined in the previous subsection, then  we set $f_{\kappa} := f \one_{\kappa}$. The main result of this section is the following.
 \begin{proposition}
 Let $d\geq 3$ and $\frac{2(d+2)}{d} \leq p\leq  \frac{2(d+1)}{d-1}$. Then there exists an exponent $1\leq s<2$, which can be taken arbitrarily close to $2$, for which the following bilinear extension estimates hold, uniformly in $N,r,f,g$. Let $f,g\in L^2(\R^d)$ be admissible functions, and let $N\geq 1$ be a dyadic number.

\begin{enumerate}
\item[(i)] If $0<r\leq 1$ is a dyadic number, and $\kappa\sim\kappa'\in\mathcal{D}_{N,r}$, then
\begin{equation}\label{Bilinear1}
\|T(f_\kappa) T(g_{\kappa'})\|_{L^{\frac p2}(\R^{d+1})}\lesssim_s N^{-\frac2s} r^{\frac{2d}{s'}-\frac{2(d+2)}{p}} \|f_\kappa\|_{L^s(\R^d)}\|g_{\kappa'}\|_{L^s(\R^d)}.
\end{equation}

\item[(ii)] If $1<r\leq N$ is a dyadic number, and $\kappa\sim\kappa'\in\mathcal{D}_{N,r}$, then
 \begin{equation}\label{Bilinear2}
\|T(f_\kappa)  T(g_{\kappa'})\|_{L^{\frac p2}(\R^{d+1})}\lesssim_s N^{-\frac2s} r^{\frac{2(d-1)}{s'}-\frac{2(d+1)}{p}} \|f_\kappa\|_{L^s(\R^d)}\|g_{\kappa'}\|_{L^s(\R^d)}.
\end{equation}
\end{enumerate}
 \end{proposition}

\begin{proof}
We first establish the estimate in the elliptic regime $0<r\leq1$. 
The proof consists of a rescaling of the bilinear extension result of Tao \cite{Ta03}.
We start by constructing affine transformations that map separated caps $\kappa\sim\kappa'\in\mathcal{D}_{N,r}$ into unit separated regions.

\smallskip

{\it Boosted caps.}
Let $N\geq 1$ and $0<r\leq 1$ be dyadic numbers, and let $\kappa\sim\kappa'\in\mathcal{D}_{N,r}$.
Let $\widetilde{\kappa}, \widetilde{\kappa}'$ denote the {\it lifts} of the caps $\kappa,\kappa'$ into the hyperboloid $\mathbb{H}^d$, defined as 
\begin{equation}\label{lift_def}
\widetilde{\kappa}=\{(\xi,\jp{\xi}):\xi\in\kappa\},\qquad \widetilde{\kappa}'=\{(\xi,\jp{\xi}):\xi\in\kappa'\}.
\end{equation}
Let $\xi_0=c(\kappa)$ denote the center of the cap $\kappa$ as in \eqref{eq:DefCenter},
and let $L_{\xi_0}$ be the Lorentz transformation defined in \eqref{eq:Lorentz} with $\nu=\xi_0$.
Then $L_{\xi_0}$ maps $\widetilde{\kappa}, \widetilde{\kappa}'$ into the lifts $\widetilde{\lambda}, \widetilde{\lambda}'$ of 
sets  $\lambda:=L_{\xi_0}^\flat (\kappa)$ and $\lambda':=L_{\xi_0}^\flat (\kappa')$ which are contained in $r$-separated cubes of sidelength comparable to $r$.
Moreover, we can take the center of the cube containing $\lambda$ to be $L^\flat_{\xi_0}(\xi_0)=0$.
Recall that the Lorentz boost $L_{\xi_0}$ is volume preserving, $\det (L_{\xi_0})=1$.
Moreover, on $\kappa\cup\kappa'$, the map $L_{\xi_0}^\flat$ 
has Jacobian determinant $\det (DL_{\xi_0}^\flat)\simeq N^{-1}$.

\smallskip

{\it Parabolic rescaling.}
The region $\{(\xi,\jp{\xi})\in\R^d\times\R: |\xi|\lesssim 1\}$ is of elliptic type, in the terminology of \cite[Section 9]{Ta03}.
The parabolic rescaling
$$P_r(\xi,\tau):=\big(\tfrac\xi r,\tfrac{\tau-1}{r^2}\big),$$
maps the lifts $\widetilde{\lambda}, \widetilde{\lambda}'$ defined above to the lifts $\widetilde{\rho},\widetilde{\rho}'$ into the compact hypersurface
\begin{equation}\label{eq:Sigmar}
\Sigma_r:=\big\{\big(\xi,\tfrac{\jp{r\xi}-1}{r^2}\big): |\xi|\lesssim1\big\}
\end{equation}
of $O(1)$-separated sets $\rho,\rho'$ of diameter comparable to 1.
Let $P_r^\flat:\R^d\to\R^d$ denote the map $\xi\mapsto r^{-1}{\xi}$,
whose Jacobian determinant satisfies $\det (D P_r^\flat)=r^{-d}$.
Note that $P^\flat_{r^{-1}}\circ P^\flat_r=\text{Id}$, and that $P_r$ is an affine map whose linear part has determinant equal to $r^{-(d+2)}$.

\smallskip

{\it Bilinear extension of caps.}
With $\rho,\rho'$ as defined above, set $f_\rho:=f\one_\rho$ and $g_{\rho'}:=g\one_{\rho'}$.
Let $\mathcal{E}_r$ denote the Fourier extension operator associated to the hypersurface $\Sigma_r$ defined in \eqref{eq:Sigmar},
$$\mathcal{E}_r (f)(x,t):=\int_{\R^d} e^{ix\cdot\xi} e^{it\Phi_r(\xi)} f(\xi)\,\d\xi,$$
with phase function given by $\Phi_r(\xi):=\frac{\jp{r\xi}-1}{r^2}$.
The hypersurfaces $\{\Sigma_r\}_{0<r\leq 1}$ are uniformly elliptic in the sense of \cite{TVV98}.
As a consequence of Tao's bilinear extension theorem for general elliptic hypersurfaces \cite[Section 9]{Ta03},
the  estimate 
\begin{equation}\label{eq:TaoBilinear}
\|\mathcal{E}_r (f_\rho)\mathcal{E}_r (g_{\rho'})\|_{L^q}
\lesssim \|f_\rho\|_{L^2}\|g_{\rho'}\|_{L^2},\qquad q>\tfrac{d+3}{d+1},
\end{equation}
holds, uniformly in $0<r\leq1$.
Using the Riesz--Th\"orin convexity theorem to interpolate the latter inequality with the trivial estimate
\begin{equation}\label{eq:trivial}
\|\mathcal{E}_r (f_\rho)\mathcal{E}_r (g_{\rho'})\|_{L^\infty}
\leq \|f_\rho\|_{L^1}\|g_{\rho'}\|_{L^1},
\end{equation}
we conclude the existence of $s_0<2$, such that
\begin{equation}\label{eq:PreChange}
\|\mathcal{E}_r (f_\rho)\mathcal{E}_r (g_{\rho'})\|_{L^{\frac p2}}
\lesssim \|f_\rho\|_{L^s}\|g_{\rho'}\|_{L^s},
\end{equation}
for every $s\in(s_0,2)$.
We claim that \eqref{Bilinear1} follows from \eqref{eq:PreChange} by a standard change of variables, which we now present in detail. 
Start by noting that 
$f_\rho=f_\kappa\circ L_{-\xi_0}^\flat\circ P_{r^{-1}}^\flat.$
It follows that
\begin{equation}\label{eq:Change1}
\|f_\rho\|_{L^s}^s
=\int_{\rho} |f_\kappa(L_{-\xi_0}^\flat\circ P_{r^{-1}}^\flat(\xi))|^s \d\xi
\simeq (Nr^d)^{-1}\|f_\kappa\|_{L^s}^s,
\end{equation}
since on $\kappa$ the change of variables $\xi=(P_r^\flat\circ L_{\xi_0}^\flat)(\zeta)$ has Jacobian determinant comparable to 
$$\det (DP_r^\flat)\det (DL_{\xi_0}^\flat)\simeq r^{-d}N^{-1}.$$
On the other hand, a straightforward computation shows that
$$\mathcal{E}_r (f_\rho)(x,t)=r^{-d}e^{-i\frac{t}{r^2}} \int_{\R^{d+1}} e^{i(\frac xr,\frac t{r^2})\cdot(\xi,\tau)} f_\kappa(L_{-\xi_0}(\xi,\tau)) \ddirac{\tau^2-\jp{\xi}^2}\jp{\xi} \,\d \xi\,\d \tau,$$
and so another change of variables $L_{-\xi_0}(\xi,\tau)=(\xi',\tau')$ yields
$$\mathcal{E}_r (f_\rho)(x,t)=r^{-d}e^{-i\frac{t}{r^2}} \int_{\R^{d}} e^{iL_{\xi_0}^T(\frac xr,\frac t{r^2})\cdot(\xi,\jp{\xi})} f_\kappa(\xi)\jp{L_{\xi_0}^\flat(\xi)} \frac{\d \xi}{\jp{\xi}}.$$
This in turn can be rewritten as
$$\mathcal{E}_r (f_\rho)(x,t)=r^{-d}e^{-i\frac{t}{r^2}}T(f_\kappa{\jp{L_{\xi_0}^\flat(\cdot)}})(L_{\xi_0}^T(\tfrac xr,\tfrac t{r^2})),$$
and so, in particular,
$$|\mathcal{E}_r (f_\rho)(x,t)|=r^{-d}|T(f_\kappa{\jp{L_{\xi_0}^\flat(\cdot)}})((L_{\xi_0}^T\circ D_r)(x,t))|,$$
where $D_r$ denotes the parabolic dilation $D_r(x,t):=(\frac xr, \frac t{r^2})$.
It follows that
\begin{equation}\label{eq:Change2}
\big\|\mathcal{E}_r (f_\rho) \mathcal{E}_r (g_{\rho'})\big\|_{L^{\frac p2}}^{\frac p2}=r^{-dp} r^{d+2} \big\|T(f_\kappa{\jp{L_{\xi_0}^\flat(\cdot)}}) T(g_{\kappa'}{\jp{L_{\xi_0}^\flat(\cdot)}})\big\|_{L^{\frac p2}}^{\frac p2}.
\end{equation}
Since $\jp{L_{\xi_0}^\flat(\xi)}\simeq 1$ if $\xi \in\kappa\cup\kappa'$, inequality \eqref{Bilinear1} is now easily seen to follow from \eqref{eq:PreChange}, \eqref{eq:Change1} and \eqref{eq:Change2}. This concludes the verification of the elliptic case.

\smallskip

For the conic case $1<r\leq N$, we can follow a similar path, invoking either Wolff's bilinear estimates for the cone \cite{Wo01} or a variant on Tao's estimates for the paraboloid noted in \cite{LeeVargas kflat}.
We choose to take a shortcut, noting that Candy's recent work \cite{Ca17} on bilinear restriction estimates for general phases already implies the adequate rescaled substitute of \eqref{eq:TaoBilinear} in the conic regime.
More precisely, \cite[Theorem 1.10]{Ca17} specializes to the inequality
\begin{equation}\label{eq:FromTim}
\|T(f_\kappa)T(g_{\kappa'}) \|_{L^q}\lesssim N^{-1}r^{d-1-\frac{d+1}q}\|f_\kappa\|_{L^2}\|g_{\kappa'}\|_{L^2},\;\;\;q>\tfrac{d+3}{d+1}.
\end{equation}
As before, this can be interpolated with the trivial 
$$\|T(f_\kappa)T(g_{\kappa'}) \|_{L^\infty}\lesssim N^{-2}\|f_\kappa\|_{L^1}\|g_{\kappa'}\|_{L^1}$$
to yield \eqref{Bilinear2}.
The proof is now complete.
\end{proof}

%%%%%%%%%%%%%%%%%%%%%%%%%%%%%%%%%%%%%%%%%%%%%%%%%%%%%%%%%%%%%%%%%%%%%%%%%%%%%%%%%%%%%%%%%%%%%%%%%%%%%%%%%%%%%%%%%%%%%%%%%%%%%%%%%%%%%%%%%%%%%%%%

\section{A refined Strichartz estimate} \label{sec:RefinedStrichartz}
There exists a well-established program, using tools from Littlewood--Paley theory, Whitney-type decompositions and quasi-orthogonality, to derive refined inequalities of Strichartz type from bilinear restriction estimates, see for instance the works \cite{BV07, KV13, KSV12, Ra12}.

The goal of this section is to establish the following refinement of inequality \eqref{ExtensionInequality} which holds for admissible functions in each dyadic annulus. 
\begin{theorem}\label{Sharpened}
Let $d\geq 3$ and $\frac{2(d+2)}{d} \leq p\leq  \frac{2(d+1)}{d-1}$.
Then there exists $\gamma\in (0,1-\frac2p)$ such that the following inequality holds
\begin{align}
\|T (f_N)\|_{L^p(\R^{d+1})}^p
\lesssim&
\Bigg[
\sup_{0<r\leq 1}(r^d)^{(\frac p2-\frac{d+2}{d})(1-\gamma)} \Big(\sup_{\kappa\in\mathcal{D}_{N,r}} 
 \|T (f_\kappa)\|_{L^p(\R^{d+1})}^{p\gamma}\Big)\notag\\
&+
\sup_{1<r\leq N} (r^{d-1})^{(\frac p2-\frac{d+1}{d-1})(1-\gamma)}\Big(\sup_{\kappa\in\mathcal{D}_{N,r}} 
 \|T (f_\kappa)\|_{L^p(\R^{d+1})}^{p\gamma}\Big)
\Bigg]
\|f_N\|_{L^2(\H^d)}^{p(1-\gamma)},\label{sharpenedextensionineq}
\end{align}
for every dyadic number $N\geq 1$ and admissible function $f\in L^2(\H^d)$.
\end{theorem}

\noindent {\it Remark.} 
Both exponents in $r$ appearing on the right-hand side of inequality \eqref{sharpenedextensionineq} are favorable:
$\frac p2-\frac{d+2}{d}\geq 0$ (in case $r\leq1$) and
$\frac p2-\frac{d+1}{d-1}\leq 0$ (in case $r>1$), with strict inequality except for the case of endpoint exponents.\\

We start with two technical lemmata which bound certain quantities that will naturally appear in the course of the proof of Theorem \ref{Sharpened}. 

\begin{lemma}\label{lem:almostorthogonal}
Let $d\geq 3$ and $\frac{2(d+2)}{d} \leq p\leq  \frac{2(d+1)}{d-1}$.
Then the following inequality holds
\begin{equation}\label{pstarineq}
\Big\|\sum_{0<r\leq N}\sum_{\kappa\sim\kappa'\in\mathcal{D}_{N,r}} 
T(f_\kappa) T(f_{\kappa'})\Big\|_{L^{\frac p2}(\R^{d+1})}^{\frac p2}
\lesssim
\sum_{0<r\leq N}\sum_{\kappa\sim\kappa'\in\mathcal{D}_{N,r}} \big\|T (f_\kappa) T (f_{\kappa'})\big\|_{L^{\frac p2}(\R^{d+1})}^{\frac p2},
\end{equation}
for every dyadic number $N\geq 1$ and  admissible function $f\in L^2(\H^d)$.
\end{lemma}

\begin{proof}

Let $\kappa\in\mathcal{D}_{N,r}$ be given, and let $\xi_0=c(\kappa)$ denote its center as in \eqref{eq:DefCenter}.
For every $\xi\in\kappa$, one easily checks that
\begin{gather}
\big| |\xi|- |\xi_0| \big|\lesssim \min\{1,r\} N,\label{eq:Property1}\\
\big(|\xi||\xi_0|-\xi\cdot\xi_0\big)^{\frac12}\lesssim r.\label{eq:Property2}
\end{gather}
Indeed, inequality \eqref{eq:Property1} follows from the fact that the length along the radial direction 
of $r$-caps and $r$-sectors at scale $N$ is comparable to $rN$ and to $N$, respectively,
and  inequality \eqref{eq:Property2} amounts to the fact that the angle between the vectors $\xi$ and $\xi_0$ is $O(\frac rN)$.
Now, given $\kappa\sim\kappa'\in\mathcal{D}_{N,r}$, with corresponding centers $\xi_0=c(\kappa)$ and $\xi_0'=c(\kappa')$, the following estimate follows from the definition of the separation relation $\sim$:
\begin{equation}\label{eq:Property3}
\frac{\big||\xi_0|-|\xi_0'|\big|}{N^2}+\frac{\big(|\xi_0||\xi_0'|-\xi_0\cdot \xi_0'\big)^{\frac12}}{N}\simeq\frac rN.
\end{equation}
Let $\widetilde{\kappa}$ and $\widetilde{\kappa}'$ be the lifts of the regions $\kappa$ and $\kappa'$ into the hyperboloid $\mathbb{H}^d$ as defined in \eqref{lift_def}. 

\smallskip

We aim to use \cite[Lemma 2.2]{Ra12} (which is a slightly more general version of \cite[Lemma A.9]{KV13} and \cite[Lemma 6.1]{TVV98}) to obtain the quasi-orthogonality proposed in \eqref{pstarineq}. Our first task is to understand the geometry of the sumset
$$\widetilde{\kappa} + \widetilde{\kappa}' = \big\{ (\xi+\xi',\jp{\xi}+\jp{\xi'}): (\xi,\xi')\in\kappa\times\kappa'\big\}\subset\R^{d+1}.$$
Using \eqref{eq:Property1}, \eqref{eq:Property2} and \eqref{eq:Property3}, one may reason as in \cite[Proof of Prop. 15]{COSS17} to further check that\footnote{Here we use the notation $\jp{x}_s:=(s^2+|x|^2)^{\frac12}$. Estimates \eqref{eq:vertical}--\eqref{eq:angular} also appear in \cite[Proof of Theorem 2.6]{Ca17}.}
\begin{gather}
\jp{\xi}+\jp{\xi'} - \jp{\xi+\xi'}_2 \simeq\tfrac{r^2}N,\label{eq:vertical}\\
\big||\xi+\xi'|-|\xi_0+\xi_0'|\big|\lesssim\min\{1,r\}N,\label{eq:radial}\\
\big(|\xi+\xi'| |\xi_0+\xi_0'| - (\xi+\xi')\cdot (\xi_0+\xi_0')\big)^{\frac12}\lesssim r.\label{eq:angular}
\end{gather}

\smallskip

\noindent {\it Step 1.} Observe that \eqref{eq:vertical}, \eqref{eq:radial} and \eqref{eq:angular} imply that the sumsets $\widetilde{\kappa} + \widetilde{\kappa}'$ are {\it almost disjoint}, in the following sense: There exists a universal constant such that, for any pair $(\kappa,\kappa')$ with $\kappa\sim\kappa'\in\mathcal{D}_{N,r}$, the number of pairs $(\rho,\rho')$ with $\rho\sim\rho'\in\mathcal{D}_{N,s}$ and 
\begin{equation}\label{non-empty-intersection}
\big(\widetilde{\kappa} + \widetilde{\kappa}') \cap \big(\widetilde{\rho} + \widetilde{\rho}') \neq \emptyset
\end{equation}
is bounded by this constant. In fact, if \eqref{non-empty-intersection} occurs, then estimate \eqref{eq:vertical} implies the existence of universal constants $a, b \in \Z$ such that $2^a r \leq s \leq 2^b r$. Let $\eta_0 = c(\rho)$ denote the center of $\rho$. 
Once $s$ is trapped, then \eqref{eq:Property1}, \eqref{eq:Property3} and \eqref{eq:radial}  imply that the lengths of $|\eta_0|$ and $|\xi_0|$ are not far from each other,  in the sense that 
\begin{equation}\label{modulus_not_so_far}
\big| |\eta_0|- |\xi_0| \big|\lesssim\min\{1,r\}N.
\end{equation}
In a similar way, \eqref{eq:Property2}, \eqref{eq:Property3} and \eqref{eq:angular} together imply that the angle between $\eta_0$ and $\xi_0$ is controlled, that is
\begin{equation}\label{angle_not_so_far}
\big(|\eta_0||\xi_0|-\eta_0\cdot \xi_0\big)^{\frac12} \lesssim r.
\end{equation}

Expressions \eqref{modulus_not_so_far} and \eqref{angle_not_so_far} imply that, given $\xi_0$, the number of possible choices for $\eta_0$ in the dyadic decomposition is finite and universally bounded. For each possible $\eta_0 = c(\rho)$, the number of regions $\rho'$ separated from $\rho$ is also finite and universally bounded.

\smallskip

\noindent {\it Step 2.} Observe that \begin{equation}\label{eq:suppFT}
\text{supp } \mathcal{F}_{t,x}[T(f_\kappa) T(f_{\kappa'})]\subset \widetilde{\kappa}+\widetilde{\kappa}',
\end{equation}
where $\mathcal{F}_{t,x}$ denotes the space-time Fourier transform. In order to use \cite[Lemma 2.2]{Ra12}, it is convenient to place the sumsets $\widetilde{\kappa}+\widetilde{\kappa}'$ inside regions which are geometrically simpler but still almost disjoint. Expression \eqref{eq:vertical} already implies that
\begin{equation}\label{eq:defGammakk}
\widetilde{\kappa}+\widetilde{\kappa}' \subset \big\{(\xi,\tau) \in \R^{d}\times\R : \jp{\xi}_2 + c_1 \tfrac{r^2}{N}\leq \tau \leq \jp{\xi}_2 + c_2 \tfrac{r^2}{N}  \ \ {\rm and} \ \ \xi \in \kappa+\kappa'\big\} =: \Gamma_{\kappa,\kappa'},
\end{equation}
for some universal constants $c_1, c_2$.  Note that equations \eqref{eq:radial} and \eqref{eq:angular} imply that the set $\kappa + \kappa'$ lies inside a rectangle centered at $\gamma_0:=\xi_0 + \xi'_0$, of height comparable to $\min\{1,r\}N$ (the major axis being aligned with the vector $\gamma_0$) and of  sidelength comparable to $r$. Denote this rectangle by $R_{\kappa,\kappa'}$.
Consider a centered dilation\footnote{More generally, given a parallelepiped $P$ and  $\lambda>0$, we denote by $\lambda\cdot P$ the {\it centered dilate} of $P$.
In other words, if $c_P$ denotes the center of $P$, then $\lambda\cdot P:=\lambda(P-c_P)+c_P$.} 
$R_{\kappa,\kappa'}^*:=(1+\alpha)\cdot R_{\kappa,\kappa'}$  
of $R_{\kappa,\kappa'}$, with $\alpha>0$ sufficiently small and independent of $(\kappa,\kappa')$,  such that the sets
\begin{equation*}
\Sigma_{\kappa,\kappa'}:=\big\{(\xi,\tau) \in \R^{d}\times\R : \jp{\xi}_2 + \tfrac{c_1}2 \tfrac{r^2}{N}\leq \tau \leq \jp{\xi}_2 + 2c_2 \tfrac{r^2}{N}  \ \ {\rm and} \ \ \xi \in R_{\kappa,\kappa'}^*\big\}
\end{equation*}
still have bounded overlap. We may now decompose the collection $\{(\kappa,\kappa'):\kappa\sim\kappa'\}$ as a union of a finite (universal) number of subsets whose corresponding $\{\Sigma_{\kappa,\kappa'}\}$ are pairwise disjoint.
By the triangle inequality, it suffices to bound the sum over just one of these subsets, which we henceforth denote by $\mathcal{T}$.

\smallskip
\noindent {\it Step 3.} We claim the existence of a universal number $K$ with the following property:  
For every $(\kappa,\kappa')\in\mathcal{T}$,
there exist  parallelepipeds $\{P_\ell=P_\ell(\kappa,\kappa')\}_{\ell=1}^K$ with disjoint interiors, satisfying
\begin{equation*}
    \Gamma_{\kappa,\kappa'}\subset \bigcup_{\ell=1}^K P_\ell,
\end{equation*}
and such that $(1+\beta)\cdot P_\ell\subset \Sigma_{\kappa,\kappa'}$, for some universal $\beta>0$.

\smallskip
Indeed, given a point $\gamma\in \R^d$, define $T(\gamma)$ to be the tangent plane to the hyperboloid $\H^d_2$ at the point $(\gamma, \langle \gamma \rangle_2)$, i.e.
$$T(\gamma):=\{(\gamma,\langle\gamma\rangle_2)+v:\, v \in \R^{d+1},\, v\bot (\gamma/\langle\gamma\rangle_2, -1)\}.$$ 
Let $e_1, e_2, \ldots, e_{d+1}$ denote the canonical basis vectors in $\R^{d+1}$. Without loss of generality, assume  $\gamma_0$ to be parallel to $e_d$. At a vector $te_d$, the slope of the tangent to the hyperbola $\{(te_d, \jp{t}_2): t\in\R\}$ equals $t/\jp{t}_2$. We may then consider a point $\gamma = te_d$ sufficiently close to $\gamma_0$, and  the corresponding hyperplane 
$$T(\gamma)=\{(\gamma, \jp{\gamma}_2) + (x_1, x_2, \ldots, x_{d-1}, x_d, x_d \,t/\jp{t}_2) \  \textrm{with each} \ x_i \in \R \}.$$ Lifting the rectangle $R_{\kappa,\kappa'}$ to the hyperplane $T(\gamma)$  amounts to choosing $|(x_1, x_2, \ldots ,x_{d-1})| \lesssim r$ and $x_d \simeq \min\{1,r\}N$. Set $y=(x_1, x_2, \ldots ,x_{d-1})$, and assume  
\begin{equation}\label{Cond_1_Rectangles}
|y| \leq c_3 r 
\text{ and  } 
|x_d| \leq c_4 \min\{1,r\}N,
\end{equation} 
for some constants $c_3, c_4$ which are yet to be  chosen. Under these assumptions, we may estimate the largest displacement in the vertical direction $e_{d+1}$ between the hyperplane $T(\gamma)$ and the hyperboloid $\H_2^d$ as follows. 
Recalling that $t \simeq N$, this displacement is given by 
\begin{equation*}
\sqrt{4 + (t+x_d)^2 + |y|^2} - \left( \sqrt{4 + t^2} + \frac{t x_d}{\sqrt{4 + t^2}}\right) \simeq \frac{4x_d^2 + |y|^2(4 +t^2)}{N^3}. 
\end{equation*}
By choosing the constants $c_3, c_4$ sufficiently small (but universal), we can bound this displacement from above by $\delta\frac{c_1}{10}\frac{r^2}{N}$, where $\delta<1$ is chosen so $c_2+\frac{\delta c_1}{10}<\frac{3c_2}{2}$, i.e. $\delta < \frac{5c_2}{c_1}$, where $c_1, c_2$  are the universal constants appearing in the definition \eqref{eq:defGammakk} of $\Gamma_{\kappa,\kappa'}$. This implies the existence of a  constant $K$, such that the original rectangle $R_{\kappa,\kappa'}$ can be decomposed into a union of $K$ smaller rectangles $\{R_\ell=R_\ell({\kappa,\kappa'})\}_{\ell=1}^K$ of the same size having disjoint interiors and verifying conditions \eqref{Cond_1_Rectangles}. 
We again emphasize that, once $c_3$ and $c_4$ are chosen, the number $K$ is universal.

\smallskip
For each $\ell$, let $\alpha_\ell$  be the center of the rectangle $R_\ell$, and let $T(R_\ell)$ denote the lift of $R_\ell$ into the hyperplane $T(\alpha_\ell)$. Define the region $P_\ell = P_\ell(\kappa,\kappa')  \subset \R^{d+1}$ as the sumset 
$$P_\ell := T(R_\ell) + \{s e_{d+1} : c_1 \tfrac{r^2}{N} \leq s \leq (c_2 + \tfrac{\delta c_1}{10}) \tfrac{r^2}{N}\}.$$
 Note that each $P_\ell$ is a parallelepiped lying above the hyperboloid $\H^d_2$ of height comparable to $ r^2/N$.
 Moreover, distinct elements of the family $\{P_\ell\}_{\ell=1}^K$ have disjoint interiors. Further observe that 
\begin{equation}\label{eq:support_gamma}
    \widetilde{\kappa}+\widetilde{\kappa}' \subset \Gamma_{\kappa,\kappa'} \subset \bigcup_{\ell=1}^{K} P_\ell.
\end{equation}
It follows from the construction of $R_{\kappa,\kappa'}^*$ and $\{R_\ell\}$ that there exists $\beta>0$, such that $(1+\beta)\cdot R_\ell \subset R_{\kappa,\kappa'}^*$, for every $\ell\in\{1,2,\ldots,K\}$. From the aforementioned displacement considerations and the choice of $\delta$ (by possibly choosing a smaller  $\beta$, depending only on $c_1,c_2$), we may guarantee that the parallelepipeds $\{P_\ell\}$  further satisfy 
\begin{equation}\label{eq:dilation inclusion}
    (1+\beta)\cdot P_\ell\subset \Sigma_{\kappa,\kappa'}.
\end{equation}
This concludes the the verification of claim.
\smallskip

\noindent {\it Step 4.}  Define $\psi_\ell:=\mathbbm{1}_{P_\ell}$.
The estimate 
\begin{equation}\label{eq:parallelepid_boundedness}
\|f\ast \widehat{\psi}_\ell\|_{L^q(\R^{d+1})}    \leq C\|f\|_{L^q(\R^{d+1})},
\end{equation} 
which holds for any exponent $q>1$, follows from a simple application of the boundedness of the Hilbert transform, 
yielding a constant $C=C_{q,d}<\infty$ that does not depend 
 on $\ell$ nor on $(\kappa,\kappa')$. 
 By the support considerations from \eqref{eq:suppFT} and \eqref{eq:support_gamma}, we have that
$$T(f_\kappa)T(f_{\kappa'})=\sum_{\ell=1}^K(T(f_\kappa)T(f_{\kappa'}))\ast \widehat{\psi}_\ell.$$
\smallskip
By the triangle inequality, it suffices to establish the estimate
\begin{equation}\label{eq:almost_there}
\Big\|\sum_{({\kappa,\kappa'})\in\mathcal{T}} 
(T(f_\kappa)T(f_{\kappa'}))\ast \widehat{\psi}_\ell\Big\|_{L^{\frac p2}(\R^{d+1})}^{\frac p2}
\lesssim
\sum_{({\kappa,\kappa'})\in\mathcal{T}}  \big\|T (f_\kappa) T (f_{\kappa'})\big\|_{L^{\frac p2}(\R^{d+1})}^{\frac p2},
\end{equation}
for each $\ell\in\{1,2,\ldots,K\}$.
From  \cite[Lemma 2.2]{Ra12}, we have that
\begin{equation}\label{eq:almost_there2}
\Big\|\sum_{({\kappa,\kappa'})\in\mathcal{T}}
(T(f_\kappa)T(f_{\kappa'}))\ast \widehat{\psi}_\ell\Big\|_{L^{\frac p2}(\R^{d+1})}^{\frac p2}
\lesssim
\sum_{({\kappa,\kappa'})\in\mathcal{T}}  \big\|(T(f_\kappa)T(f_{\kappa'}))\ast \widehat{\psi}_\ell\big\|_{L^{\frac p2}(\R^{d+1})}^{\frac p2}.
\end{equation}
In fact, the Fourier transform of each function $(T(f_\kappa)T(f_{\kappa'}))\ast \widehat{\psi}_\ell$ is supported in $P_\ell$, and as we have seen there exists $\beta>0$ such that the elements in the family $\{(1+\beta)\cdot P_\ell\}_{(\kappa,\kappa')\in\mathcal{T}}$ are pairwise disjoint. Moreover, for each $(\kappa,\kappa')\in\mathcal{T}$, one easily constructs a function $\varphi=\varphi(\kappa,\kappa')$ satisfying 
\begin{align*}
\begin{split}
    \supp{(\varphi)}&\subset (1+\beta)\cdot P_\ell, \\
    \varphi(x)&\equiv 1,\, \mathrm{if} \,\,x\in P_\ell, \\
    \|\widehat{\varphi}\|_{L^1(\R^{d+1})}&\leq C,
\end{split}
\end{align*}
where the constant $C$ is uniform in $(\kappa,\kappa')$.
One just has to observe that each parallelepiped $P_\ell$ is an affine image of the unit cube.
Therefore \eqref{eq:almost_there2} follows from a direct application of \cite[Lemma 2.2]{Ra12}.
To finish, invoke \eqref{eq:parallelepid_boundedness} with $q=p/2>1$ in order to obtain \eqref{eq:almost_there} from \eqref{eq:almost_there2}.
The proof is now complete.
\end{proof}

\begin{lemma}\label{BegoutVargasLemma}
Let $d\geq 3$ and $\frac{2(d+2)}{d} \leq p\leq  \frac{2(d+1)}{d-1}$.
Let $1\leq s<2$ and $0<\gamma<1-\frac2p$.
Then the following inequality holds
\begin{equation}\label{BegoutVargas}
\sum_{0<r\leq N}\sum_{\kappa\in\mathcal{D}_{N,r}} \big(|\kappa|^{1-\frac2s} \|f_\kappa\|_{L^s(\R^d)}^2\big)^{\frac p2(1-\gamma)}\lesssim \|f_N\|_{L^2(\R^d)}^{p(1-\gamma)},
\end{equation}
for every dyadic number $N\geq 1$ and admissible function $f\in L^2(\H^d)$.
\end{lemma}

\begin{proof}
We may assume that $\|f_N\|_{L^2} = 1$.  The strategy, suggested by the proof of \cite[Theorem 1.3]{BV07}, amounts to decomposing the function $f_N$ into low and high frequencies, depending on the size of the region $\kappa$.
More precisely, write
$$f_N
=\one_{\{|{f}|\leq |\kappa|^{-\frac12}\}}{f_N}+\one_{\{|{f}|> |\kappa|^{-\frac12}\}}{f_N}
=:{f_N^{\leq}}+{f_N^{>}}.$$
Set $\alpha:=\frac p2(1-\gamma)$.
To estimate the low frequencies, use H\"older's inequality to bound  
$$ \|f_N^{\leq}\|_{L^s(\kappa)}
\leq 
|\kappa|^{\frac1s-\frac1{2\alpha}}\|f_N^{\leq}\|_{L^{2\alpha}(\kappa)},$$
which holds provided $2\alpha\geq s$, or equivalently $\gamma\leq1-\frac sp$. 
In this case,
\begin{equation} \label{E:sum fN leq}
\sum_{0<r\leq N}\sum_{\kappa\in\mathcal{D}_{N,r}} 
\big(|\kappa|^{1-\frac2s} \|f_N^{\leq}\|_{L^s(\kappa)}^2\big)^{\alpha}\lesssim 
\sum_{0<r\leq N}\sum_{\kappa\in\mathcal{D}_{N,r}} 
|\kappa|^{\alpha-1}\|f_N^{\leq}\|^{2\alpha}_{L^{2\alpha}(\kappa)}.
\end{equation}
Let $V_{N,r}$ denote the volume of a region $\kappa \in \mathcal D_{N,r}$.  Recall that $V_{N,r} \simeq Nr^{d}$ when $0 < r \leq 1$, and that $V_{N,r} \simeq Nr^{d-1}$ when $1 < r < N$.
The right-hand side of \eqref{E:sum fN leq} can be estimated as follows:
\begin{equation}\label{geom_series}
\sum_{0<r\leq N}\sum_{\kappa\in\mathcal{D}_{N,r}} 
|\kappa|^{\alpha-1}\|f_N^{\leq}\|^{2\alpha}_{L^{2\alpha}(\kappa)}
\lesssim
\int_{|\xi|\simeq N}
\Bigg(\sum_{\substack{0<r\leq N: \\ |f(\xi)|\leq (V_{N,r})^{-\frac12}}}
(V_{N,r})^{\alpha-1}\Bigg)
 |f(\xi)|^{2\alpha} \d\xi.
\end{equation}
 Thus the sum on the right-hand side of \eqref{geom_series} amounts to  two geometric series, both of which can be estimated by their largest terms:
$$\int_{|\xi|\simeq N}
\Bigg(\sum_{\substack{0<r\leq N: \\ |f(\xi)|\leq (V_{N,r})^{-\frac12}}}
(V_{N,r})^{\alpha-1}\Bigg)
 |f(\xi)|^{2\alpha} \d\xi
 \lesssim
 \int_{|\xi|\simeq N} |f(\xi)|^{-2(\alpha-1)}|f(\xi)|^{2\alpha}\d\xi=\|f_N\|_{L^2}^2=1.$$
Note that the latter inequality only holds provided $\alpha-1>0$, or equivalently $\gamma<1-\frac2p$, which is a valid constraint since $p>2$.

\smallskip

To estimate the high frequencies, use Minkowski's inequality to bound
$$\sum_{0<r\leq N}\sum_{\kappa\in\mathcal{D}_{N,r}} \left(|\kappa|^{1-\frac2s} \|f_N^>\|_{L^s(\kappa)}^2\right)^{\alpha}
\leq
\left(\sum_{0<r\leq N}\sum_{\kappa\in\mathcal{D}_{N,r}} |\kappa|^{\frac s2-1} \|f_N^>\|_{L^s(\kappa)}^s\right)^{\frac{2\alpha}s},
$$
which  as before holds provided $2\alpha\geq s$.
The right-hand side of this expression can be estimated as before:
\begin{align*}
\sum_{0<r\leq N}\sum_{\kappa\in\mathcal{D}_{N,r}} |\kappa|^{\frac s2-1} \|f_N^>\|_{L^s(\kappa)}^s
&\lesssim
\int_{|\xi|\simeq N}\Bigg(\sum_{\substack{0<r\leq N: \\ |f(\xi)|> (V_{N,r})^{-\frac12}}} (V_{N,r})^{\frac s2-1} \Bigg) |f(\xi)|^s \d\xi\\
&\lesssim\int_{|\xi|\simeq N} |{f}(\xi)|^{-2(\frac s2-1)} |{f}(\xi)|^s \d\xi= \|f_N\|_{L^2}^2= 1.
\end{align*}
This concludes the verification of \eqref{BegoutVargas}.
\end{proof}

We are now ready for the proof of the refined Strichartz inequality.
\begin{proof}[Proof of Theorem \ref{Sharpened}]
We recall the following simple geometric observation: 
Given dyadic numbers $N\geq 1$ and $0<r\leq N$,
and a region $\kappa\in\mathcal{D}_{N,r}$, the number of regions $\kappa' \in\mathcal{D}_{N,r}$ which are separated from $\kappa$ is universally bounded. 
In other words,
\begin{equation}\label{DoubletoSingle}
\#\{\kappa': \kappa\sim\kappa'\in\mathcal{D}_{N,r}\}\lesssim_d 1.
\end{equation}
Via a standard decomposition argument, see \cite{BV07, TVV98}, we have that
\begin{equation}\label{Step3First}
\|T(f_N)\|_{L^p}^p
=\Bigg\|\sum_{0<r\leq N}\sum_{\kappa\sim\kappa'\in\mathcal{D}_{N,r}}
T (f_\kappa) T (f_{\kappa'})\Bigg\|_{L^{\frac p2}}^{\frac p2}.
\end{equation}
To verify this, recall the definition \eqref{eq:DefAnnulus} of the restricted annulus $\mathcal{A}_N$, and consider the diagonal
$$\Gamma:=\{(\xi,\eta)\in\mathcal{A}_N\times\mathcal{A}_N: \xi=\eta\}.$$ 
Then the following Whitney-type decomposition is a consequence of the construction performed in Section \ref{sec:Caps}:
\begin{equation}\label{eq:Whitney}
(\mathcal{A}_N\times\mathcal{A}_N)\setminus\Gamma
=\bigcup_{0<r\leq N}\bigcup_{\kappa\sim\kappa'\in\mathcal{D}_{N,r}} \kappa\times\kappa'.
\end{equation}
Identity \eqref{Step3First} follows from this by writing $\|T (f_N)\|_{L^p}^p=\|T (f_N)^2\|_{L^{\frac p2}}^{\frac p2}$.
By Lemma \ref{lem:almostorthogonal}, we then have
\begin{equation}\label{eq:applyL5}
\|T (f_N)\|_{L^p}^p\lesssim
\sum_{0<r\leq N}\sum_{\kappa\sim\kappa'\in\mathcal{D}_{N,r}} \big\|T (f_\kappa) T(f_{\kappa'})\big\|_{L^{\frac p2}}^{\frac p2}.
\end{equation}
On the one hand, each of these summands can be bounded by H\"older's inequality as follows:
\begin{equation}\label{Holder}
\|T (f_\kappa) T (f_{\kappa'})\|_{L^{\frac p2}}
\leq
\|T (f_\kappa)\|^\gamma_{L^p}\|T (f_{\kappa'})\|^\gamma_{L^{p}}
\|T (f_\kappa) T (f_{\kappa'})\|^{1-\gamma}_{L^{\frac p2}},
\end{equation}
where $\gamma \in (0,1)$ is a parameter to be chosen below.
On the other hand, we can split the sum on the right-hand side of \eqref{eq:applyL5} into two pieces, depending on whether $0<r\leq 1$ or $1<r\leq N$.
Let us focus on the first sum, that over caps.
We claim that
\begin{multline}\label{AllTogether}
\sum_{0<r\leq 1}\sum_{\kappa\sim\kappa'\in\mathcal{D}_{N,r}} \big\|T (f_\kappa) T (f_{\kappa'})\big\|_{L^{\frac p2}}^{\frac p2}
\lesssim
N^{-\frac p2(1-\gamma)} 
\left(\sup_{0<r\leq 1} \sup_{\kappa\in\mathcal{D}_{N,r}} 
(r^{d})^{(\frac p2-\frac{d+2}{d})(1-\gamma)} \|T (f_\kappa)\|_{L^p}^{p\gamma}\right)\\
\times\sum_{0<r\leq 1}\sum_{\kappa\in\mathcal{D}_{N,r}} \big(|\kappa|^{1-\frac2s} \|f_\kappa\|_{L^s}^2\big)^{\frac p2(1-\gamma)}.
\end{multline}
\noindent This follows from  inequality \eqref{Holder}, 
the case $f=g$ of the bilinear extension estimate \eqref{Bilinear1},
and the observation \eqref{DoubletoSingle} that allows to bound the double sum $\sum_{\kappa\sim\kappa'}$ by a single sum $\sum_\kappa$. 
One just has to recall that the Lebesgue measure of an $r$-cap at scale $N$ is comparable to $Nr^{d}$.

Lemma \ref{BegoutVargasLemma} then implies that the last factor on the right-hand side of inequality \eqref{AllTogether} is  $O(\|f_N\|_{L^2}^{p(1-\gamma)})$, provided $\gamma<1-\frac2p$.
As a consequence, the following inequality for $r$-caps at scale $N$ holds:
\begin{multline}\label{sectors}
\sum_{0<r\leq 1}\sum_{\kappa\sim\kappa'\in\mathcal{D}_{N,r}} \big\|T (f_\kappa) T (f_{\kappa'})\big\|_{L^{\frac p2}}^{\frac p2}\lesssim
 \left(\sup_{0<r\leq 1} \sup_{\kappa\in\mathcal{D}_{N,r}} 
(r^d)^{(\frac p2-\frac{d+2}{d})(1-\gamma)} \|T (f_\kappa)\|_{L^p}^{p\gamma}\right)
\|f_N\|_{L^2(\H^d)}^{p(1-\gamma)}.
\end{multline}
In a similar way,  recalling that the Lebesgue measure of an $r$-sector at scale $N$ is comparable to $Nr^{d-1}$, 
and using \eqref{Bilinear2} instead of \eqref{Bilinear1},
one can show the corresponding inequality for $r$-sectors at scale $N$,
\begin{multline}\label{caps}
\sum_{1<r\leq N}\sum_{\kappa\sim\kappa'\in\mathcal{D}_{N,r}} \big\|T (f_\kappa) T (f_{\kappa'})\big\|_{L^{\frac p2}}^{\frac p2}
\lesssim
\left(\sup_{1<r\leq N} \sup_{\kappa\in\mathcal{D}_{N,r}} 
(r^{d-1})^{(\frac p2-\frac{d+1}{d-1})(1-\gamma)} \|T (f_\kappa)\|_{L^p}^{p\gamma}\right)
\|f_N\|_{L^2(\H^d)}^{p(1-\gamma)},
\end{multline}
under the same assumption  $\gamma<1-\frac2p$.
Inequality \eqref{sharpenedextensionineq} follows from \eqref{eq:applyL5}, \eqref{sectors} and \eqref{caps}. 
The proof is now complete.
 \end{proof}

%%%%%%%%%%%%%%%%%%%%%%%%%%%%%%%%%%%%%%%%%%%%%%%%%%%%%%%%%%%%%%%%%%%%%%%%%%%%%%%%%%%%%%%%%%%%%%%%%%%%%%%%%%%%%%%%%%%%%%%%%%%%%%%%%%%%%%%%%%%%%%%%

\section{End of the proof: concentration-compactness}\label{sec:CC}

As we left off in Section \ref{sec:ang_rest}, let $\{f_n\}_{n\in\N} \subset L^2(\mathbb{H}^d)$ be an extremizing sequence for \eqref{ExtensionInequality}, with $\|f_n\|_{L^2(\mathbb{H}^d)} = 1$ for all $n \in \N$, and let $\{f_n^{(k_0)}\}_{n\in\N}$ be a quasi-extremizing sequence in the sense of \eqref{quasi-ext}. Assuming without loss of generality that $k_0 =1$, the sequence $\{f_n^{(1)}\}_{n\in\N}$ belongs to our class of admissible functions considered in Sections \ref{sec:Caps} and \ref{sec:RefinedStrichartz}.

\smallskip 

From Proposition \ref{annulardec}, for each $n\in\N$, there exists $N = N_n\in2^{\Z_{\geq0}}$ such that 
$$\big\|T\big((f_{n}^{(1)})_N)\big)\big\|_{L^p(\R^{d+1})}\geq \delta_3,$$
where $\delta_3>0$ is a universal constant.

\smallskip

If $\frac{2(d+2)}{d} \leq p <  \frac{2(d+1)}{d-1}$, then Theorem \ref{Sharpened} ensures for each $n\in\N$ the existence of a dyadic number $r=r_n$ satisfying $r \leq  2^\alpha$ for a universal constant $\alpha$, and of a region $\kappa=\kappa_n \in \mathcal{D}_{N,r}$, such that 
$$\big\|T\big((f_{n}^{(1)})_\kappa\big)\big\|_{L^p(\R^{d+1})}\geq \delta_4,$$
where $\delta_4>0$ is a universal constant. This implies at once that 
$$\|(f_{n}^{(1)})_\kappa\|_{L^2(\mathbb{H}^{d})}\geq \delta_5,$$
where $\delta_5>0$ is a universal constant. 
Set $L_n := L_{c(\kappa)}$, where $c(\kappa)$ denotes as usual the center of the region $\kappa$. 
Since $r \leq 2^\alpha$, a standard computation shows that the image $L_n^\flat(\kappa)$ is contained in a universal ball $B\subset\R^d$ centered at the origin. 
Therefore 
$$\|L_n^*f_{n}^{(1)}\|_{L^2(B)} \geq \delta_6,$$
where $\delta_6>0$ is a universal constant. As already observed in Section \ref{sec:ang_rest}, this plainly implies that 
$$\|L_n^*f_{n}\|_{L^2(B)} \geq \delta_6.$$
This establishes the existence of distinguished region when $\frac{2(d+2)}{d} \leq p <  \frac{2(d+1)}{d-1}$.

\smallskip

We can now invoke the machinery of \cite[Section 6]{COSS17}, which only works when $\frac{2(d+2)}{d} < p \leq  \frac{2(d+1)}{d-1}$, to arrive at the existence of extremizers stated in  Theorem \ref{Thm1} in the non-endpoint range $\frac{2(d+2)}{d} < p <  \frac{2(d+1)}{d-1}$.
We provide the details below.

\smallskip

By \cite[Proposition 18]{COSS17}, there exists $(x_n,t_n)\in\R^d\times\R$ such that the sequence $\{h_n\}_{n\in\N}$ defined by 
$$h_n(\xi):=e^{ix_n\cdot \xi}e^{i t_n\langle \xi\rangle}f_n(\xi)$$
admits a subsequence that converges weakly to a nonzero limit, say $h\neq 0$, in $L^2(\H^d)$.
For this subsequence, possibly after extracting a further subsequence, \cite[Proposition 19]{COSS17}  implies 
$$T(h_n)(x,t)\to T(h)(x,t), \text{ as }n\to\infty,$$
for almost every $(x,t)\in\R^d\times\R$.
The existence of extremizers then follows from a straightforward application of  \cite[Proposition 1.1]{FVV11}.
This completes the proof of Theorem \ref{Thm1}.

%%%%%%%%%%%%%%%%%%%%%%%%%%%%%%%%%%%%%%%%%%%%%%%%%%%%%%%%%%%%%%%%%%%%%%%%%%%%%%%%%%%%%%%%%%%%%%%%%%%%%%%%%%%%%%%%%%%%%%%%%%%%%%%%%%%%%%%%%%%%%%%%

\section*{Acknowledgements}
The authors are thankful to Rupert Frank {and Ren\'e Quilodr\'an} for helpful comments and suggestions.
E.C.~acknowledges support from CNPq - Brazil, FAPERJ - Brazil and the Simons Associate Scheme from the International Centre for Theoretical Physics (ICTP) - Italy. 
D.O.S.~was partially supported by the Hausdorff Center for Mathematics, the Deutsche Forschungsgemeinschaft through the Collaborative Research Center 1060, and the College Early Career Travel Fund of the University of Birmingham. 
M.S. ~acknowledges support from the grant PICT 2014-1480 (ANPCyT).
B.S.~was supported by a National Science Foundation grant (DMS-1600458).

\end{document}